\documentclass[12pt,final]{amsart}
\usepackage{amsmath,amsthm}
\usepackage{amsfonts}
\usepackage{amssymb}
\usepackage{amscd}
\usepackage[dvipdfmx]{graphicx}
\usepackage{geometry}
\usepackage{color}
\usepackage{hyperref}

\newtheorem{theorem}{Theorem}[section]
\newtheorem{defn}{Definition}[section]
\newtheorem{lemma}{Lemma}[section]
\newtheorem{corollary}{Corollary}[section]

\newtheorem{proposition}{Proposition}[section]
\theoremstyle{definition}
\newtheorem*{remark}{Remark}
\newtheorem{example}{Example}[section]

\title[Discrete linear Weingarten surfaces with singularities]{Discrete linear Weingarten surfaces with singularities in Riemannian and Lorentzian spaceforms}                    
\author[W. Rossman]{Wayne Rossman}
\address[W. Rossman]{Department of mathematics, Faculty of science, Kobe University \newline Rokkodai-cho 1-1, Nada-ku, Kobe, 657-8501, Japan}
\email{wayne@math.kobe-u.ac.jp}
\author[M. Yasumoto]{Masashi Yasumoto}
\address[M. Yasumoto]{Department of mathematics, Faculty of science, Kobe University \newline Rokkodai-cho 1-1, Nada-ku, Kobe, 657-8501, Japan}
\curraddr[M. Yasumoto]{Institut f\"{u}r Mathematik, Universit\"{a}t T\"{u}bingen, Auf der Morgenstelle 10, 72076 T\"{u}bingen, Germany}
\email{myasu@math.kobe-u.ac.jp}

\date{}

\begin{document}
\maketitle
\normalsize

\begin{abstract}
In this paper we define and analyze singularities of discrete linear Weingarten surfaces with Weierstrass-type representations in $3$-dimensional Riemannian and Lorentzian spaceforms. In particular, we discuss singularities of discrete surfaces with non-zero constant Gaussian curvature, and parallel surfaces of discrete minimal and maximal surfaces, and discrete constant mean curvature $1$ surfaces in de Sitter 3-space, including comparisons 
with different previously known definitions of such singularities.  
\end{abstract}

\section{Introduction}

In this paper we examine discrete surfaces with Weierstrass-type representations in spaceforms, taking advantage of 
the more general setting of Lie sphere geometry and 
discrete Legendre immersions (see Definition \ref{LegLiftDefn} here), and with helpful motivations coming from the developing field of $\Omega$ surfaces. There are numerous Weierstrass-type representations in 3-dimensional spaceforms in addition to the classical representation for minimal surfaces in $\mathbb{R}^3$, for example, for 
\begin{enumerate}
\item maximal surfaces (spacelike immersion with mean curvature identically 0) in Minkowski 3-space $\mathbb{R}^{2,1}$ by Kobayashi \cite{Ko}, 
\item constant mean curvature (CMC, for short) 1 surfaces in hyperbolic 3-space $\mathbb{H}^3$ by Bryant \cite{Bry} (see also \cite{UY-CMC1}), 
\item flat surfaces in $\mathbb{H}^3$ by G\'alvez,  Mart\'{\i}nez, Mil\'an \cite{GMM-first}, 
\item CMC 1 surfaces in de Sitter 3-space $\mathbb{S}^{2,1}$ by Aiyama, Akutagawa \cite{AA}, 
\item linear Weingarten surfaces of Bryant type (BrLW surfaces, for short) in $\mathbb{H}^3$ by G\'alvez,  Mart\'{\i}nez, Mil\'an \cite{GMM-second}, 
\item linear Weingarten surfaces of Bianchi type (BiLW surfaces, for short) in $\mathbb{S}^{2,1}$ by Aledo, Espinar \cite{AE}. 
\end{enumerate}
Regarding the last two examples above, 
Izumiya and Saji \cite{IS} showed that a necessary and sufficient condition for an immersion in $\mathbb{H}^3$ to be BrLW is that its unit normal vector field is BiLW (see \S \ref{BrBi}).

Recently, there has been work on discretization of the above representations. Bobenko, Pinkall \cite{BobPink2} described discrete isothermic surfaces in the Euclidean 3-space $\mathbb{R}^3$, and as an application, they derived the Weierstrass representation for discrete isothermic minimal surfaces in $\mathbb{R}^3$, using integrable systems techniques. In the same vein,  Hertrich-Jeromin \cite{Udo1} gave the Weierstrass-type representation for discrete isothermic CMC 1 surfaces in $\mathbb{H}^3$.

Burstall, Hertrich-Jeromin and the first author \cite{BHR3} described discrete linear Weingarten surfaces in any 3-dimensional spaceform using Lie sphere geometry, which we briefly introduce in \S \ref{secfive}. 
Using that method, we can treat discrete linear Weingarten surfaces in 
any 3-dimensional spaceform.  
They did not consider singularities of discrete surfaces, however, as we will do here.  

Returning to smooth surfaces, unlike the minimal and (non-zero) CMC surfaces in  $\mathbb{R}^3$, general linear Weingarten surfaces will have singularities. In fact, singularities of maximal surfaces in $\mathbb{R}^{2,1}$, flat and BrLW surfaces in $\mathbb{H}^3$, and BiLW surfaces in $\mathbb{S}^{2,1}$ are investigated in \cite{FSUY, KU, UY-max}. Thus, it is natural to still consider singularities when discretizing these surfaces. However, difficulties occur with this (Definition \ref{defnOfSingVert}), and overcoming those difficulties is our primary task here.  

Hoffmann, Sasaki, Yoshida and the first author \cite{HRSY} described discrete 
BrLW surfaces in $\mathbb{H}^3$, and 
furthermore treated singularities of discrete flat surfaces in $\mathbb{H}^3$. For that, they considered the behavior of caustics of 
smooth flat surfaces at a singular 
point, via the Weierstrass-type representation. Such a caustic 
contacts the surface at a singular point, which 
lead to a natural definition of 
singularities in the discrete case, i.e. that a singularity of a discrete flat surface is a vertex that contacts the (discrete) caustic. We will define singular vertices in a more direct way that applies to a wider 
variety of discrete surfaces, and show equivalence of the definitions 
in the case considered just above (Theorem \ref{SingVertFlat}).  

The second author \cite{Yashi2014} described discrete maximal surfaces in the Minkowski $3$-space $\mathbb{R}^{2,1}$ and analyzed their singular faces, that is, non-spacelike faces (Definition \ref{defnOfSingFace}). This is also a natural way to define singular behavior, because the tangent plane of a smooth maximal surface is non-spacelike precisely at singular points. 

Thus, singularities of discrete surfaces could be either vertices or faces, and two of our primary results here are about relating those two viewpoints, in particular, in the cases of discrete maximal surfaces in $\mathbb{R}^{2,1}$ and discrete CMC 1 surfaces in  $\mathbb{S}^{2,1}$. 

Smooth $2$-dimensional Legendre immersions in Lie sphere geometry project to surfaces in spaceforms that can 
have singularities.  However, those surfaces considered 
together with their unit normal maps become immersions (by definition), and they are called fronts. The most typical singularities on fronts are cuspidal edges of $3/2$ type, and next perhaps are swallowtails.  At such 
singularities, exactly one of the principal curvatures will diverge (see \cite{tera2}), and equivalently, one of the principal curvature spheres will become a point sphere. Using that the notion of principal curvature spheres in Lie sphere geometry is independent of the choice of projection to a $3$-dimensional spaceform, we  define singular vertices on projections of 
discrete Legendre immersions.

While typical singularities on smooth surfaces 
can be found by locating the points where one principal curvature blows out to infinity, 
on discrete surfaces the principal curvatures are discrete functions from the set of edges to the real numbers, and thus we can only identify the vertices at which a principal curvature changes sign.  As a result, it is not so immediate to distinguish the points that are singular from the points that are parabolic (at which exactly one of the two bounded principal curvature becomes zero) or flat (at which both principal curvatures become zero). This is why we will use a particular terminology ``FPS vertices'' in Definition \ref{defnOfSingVert}. This is the first of our three goals:
\begin{enumerate}
\item We will find and examine cases where the 
distinction between singular and parabolic or 
flat points is possible.  Such cases include surfaces of 
constant Gaussian curvature (CGC) $K \neq 0$ (see \S \ref{SingVertCGC}), and some particular discrete linear Weingarten surfaces for which Weierstrass 
type representations exist (\S \ref{secseven}, \S \ref{seceight}, \S \ref{secnine}).  
\item We will confirm that the discrete Weierstrass type 
representations are compatible with other ways of defining 
discrete surfaces with specific curvature properties.  In 
particular, they are compatible with the definitions given by 
Burstall, Hertrich-Jeromin and the first author in \cite{BHR3} 
(Proposition \ref{ConsQuanBrBi}).  
\item We will find relationships between singular vertices and singular faces in particular cases (Theorem \ref{thm8pt4}, Theorem \ref{thm7pt2next}, Corollary \ref{SingFaceVert}).  
\end{enumerate}

\section{Discrete Legendre immersions}

First we recall smooth Legendre immersions in the context 
of Lie sphere geometry, that is, maps $\Lambda$ of $2$-manifolds 
$M^2$ into the collection of null planes in $\mathbb{R}^{4,2}$, 
with metric signature $(-,+,+,+,+,-)$, i.e. 
\[
(X,Y)_{\mathbb{R}^{4,2}}=(X,Y):=-x_1y_1+x_2y_2+x_3y_3+x_4y_4+x_5y_5-x_6y_6
\]
for $X=(x_1,x_2,x_3,x_4,x_5,x_6)^t, \ Y=(y_1,y_2,y_3,y_4,y_5,y_6)^t \in \mathbb{R}^{4,2}$. Then
\[
\mathbb{L}^5:=\{ X \in \mathbb{R}^{4,2} | (X,X)=0 \}
\]
denotes the light cone of $\mathbb{R}^{4,2}$.

Let $\Lambda \subset \mathbb{L}^5$ be a $2$-dimensional null subspace, which projectivizes to a line in the projectivized light cone $\mathbb{PL}^5$ called a {\em contact element}.  
This line will represent a family of spheres (a pencil) that are all tangent 
(with same orientation) at one point.

If $\Lambda$ is a (smooth) 
map from $M=M^2$ to the collection of null planes in $\mathbb{R}^{4,2}$, 
where $M$ is a $2$-dimensional 
manifold, then $\Lambda$ is a {\em Legendre immersion} if,  
\begin{enumerate}
\item for any pair of sections $X_1$, $X_2$ of $\Lambda$, 
\[
dX_1 \perp X_2 \quad \text{(contact condition), and }
\]
\item for any $m \in M$ and any choice of $Y 
\in T_m M$, $dX(Y) \in \Lambda(m)$ for all sections $X$ of $\Lambda(m)$ 
implies $Y=0$ (immersion condition).  
\end{enumerate}

The immersion condition can be 
restated in terms of a basis of sections for the null planes 
$\Lambda$ as follows: If 
\[ 
\Lambda = \text{span}\{ X_1,X_2 \} \ , 
\]
with basis $X_1,X_2: M^2 \to \mathbb{L}^5$, then the immersion condition is equivalent to 
\[ 
dX_1(Y) , dX_2(Y) \in \Lambda(m) \;\;\; \text{implies} \;\;\; Y=0 
\]
for all $Y \in T_m M$, and one can then check that this condition is independent of the choice of basis $X_1,X_2$.

By choosing two nonzero perpendicular vectors $p,q$ 
in $\mathbb{R}^{4,2}$ ($p$ not null), we can project $\Lambda$ to a surface $f: M^2 \to M^3$ in the $3$-dimensional spaceform 
\[ 
M^3 = M^3_{p,q} := \{ X \in \mathbb{R}^{4,2} \, | \, (X,X)=(X,p)=0, 
(X,q)=-1 \}  
\]
with sectional curvature $-(q,q)$, by taking $f \in \mathrm{Sec} (\Lambda)$ such that 
\begin{equation}\label{extra1}
(f,p)=0 \;\;\; \text{and} \;\;\; (f,q)=-1 \ , 
\end{equation}
where $\mathrm{Sec} (\Lambda)$ denotes the set of all sections of $\Lambda$. Note that, when we choose a constant timelike (resp. spacelike) vector $p \in \mathbb{R}^{4,2}$ and a constant vector $q \in \mathbb{R}^{4,2}$, $M^3$ becomes a $3$-dimensional Riemannian (resp. Lorentzian) spaceform. For details, see \cite{wisky-next}, and for a particular choice of $p$ and $q$, see Section \ref{BrBi}.

Let $n$ denote the unit normal to $f$ in $M^3$, i.e. $n \in \mathrm{Sec} (\Lambda)$ 
and
\[
(n,q)=0 \;\;\; \text{and} \;\;\; (n,p)=-1. 
\]
The sections of $\Lambda=\text{span}\{ f,n \}$ represent the sphere 
congruences of $f$, and then $f$, resp. $n$, is the 
point sphere, resp. tangent geodesic plane, 
congruence.  Let $s_\alpha$ for $\alpha=1,2$ be 
sections of $\Lambda$ that represent the principal 
curvature sphere congruences, which can be defined by \[ s_\alpha= 
\kappa_\alpha f + n \] using the principle curvatures $\kappa_\alpha$ 
of $f$, or equivalently by the directional derivative 
conditions that $D_{\vec v_\alpha} s_\alpha \in \mathrm{Sec} (\Lambda)$ for some tangent vector fields $\vec v_\alpha$ on $M^2$.

For $\Lambda$ above to be a Legendre immersion, both immersion and contact conditions must be satisfied.  
For a discrete Legendre map $\Lambda$ as in Definition \ref{LegLiftDefn} below, discretized versions of the immersion and contact conditions are needed. We also assume the existence of  ``discrete curvature line  coordinates'', that is, we require that the four vertices of each quadrilateral be concircular, which is called a {\em principal net}.  In this way, the properties of smooth Legendre immersions motivate the following 
definition of discrete Legendre immersions:

\begin{defn}[\cite{BHR3}] \label{LegLiftDefn}
A map 
\[
\Lambda: \mathbb{Z}^2 \; 
\text{(or some subdomain of } \mathbb{Z}^2 \text{)} \; 
\to \{ \text{null planes in } \mathbb{R}^{4,2} 
\} \]
is a {\em discrete Legendre immersion} if, for any quadrilateral, 
with vertices $i,j,k,\ell$ 
ordered counterclockwise about the quadrilateral 
and with $i$ in the lower left corner in $\mathbb{Z}^2$, and with corresponding surface vertices $f_i$, $f_j$, $f_k$, $f_l$ defined like in \eqref{extra1}, 
\begin{enumerate}
\item (principal net condition) 
$\dim (\text{span}\{ f_i,f_j,f_k,f_\ell \}) = 3$, 
\item (first immersion condition) There exist $p,q$ 
such that the difference of any two of 
$f_i,f_j,f_k,f_\ell$ is non-null, 
\item (second immersion condition) For some $p,q$ as 
in item (2) above, 
$f_k-f_i$ and $f_\ell-f_j$ are not parallel, 
\item (contact condition) $\Lambda_i \cap \Lambda_j 
\neq \{ \vec 0 \}$, $\Lambda_i \cap \Lambda_\ell 
\neq \{ \vec 0 \}$.  
\end{enumerate}
\end{defn}

\begin{remark}
Item (1) in Definition \ref{LegLiftDefn} and 
$(f_*,q)=-1$ imply 
\[
f_i,f_j,f_k,f_\ell\]
all lie in some $2$-dimensional plane.  Item (3) implies 
any two or three vertices amongst 
$f_i,f_j,f_k,f_\ell$ span a $2$ or $3$ dimensional 
subspace of $\mathbb{R}^{4,2}$, respectively, with 
nondegenerate induced metric $(+,-)$ or $(+,+,-)$.  
\end{remark}

\section{FPS vertices of projections of discrete 
Legendre immersions} 

Generically, a smooth surface (section) $f \in \mathrm{Sec} (\Lambda)$ will 
have a singularity 
when one of the 
principal curvature spheres $s_\alpha$ becomes a point 
sphere \cite{tera2}, i.e. when $s_\alpha \perp p$ for $\alpha = 1$ or 
$2$.  Also, where $f$ does not have a singularity, it will 
have a parabolic or flat point if one of the $s_\alpha$ becomes a 
tangent geodesic plane, i.e. $s_\alpha \perp q$.  

In the case of discrete Legendre immersions, 
the domain becomes $\mathbb{Z}^2$, or some subdomain of 
$\mathbb{Z}^2$, rather than $M^2$.
We define the curvature spheres as those 
spheres represented by nonzero vectors (\cite{BHR3})
\[ 
s_1 \in \Lambda_i \cap \Lambda_j \;\;\; \text{and} \;\;\; s_2 \in 
\Lambda_i \cap \Lambda_\ell \ . 
\]
Thus we have spheres in 
$M^3$, associated to edges, that lie in both of 
the sphere pencils defined at the two endpoints of 
the edges.  In particular the normal geodesics 
(i.e. the geodesics through the vertices and 
perpendicular to the spheres in the sphere pencils) 
emanating from the adjacent vertices, when they 
do intersect, will intersect at equal distances 
from the two vertices.  

Thus, $s_1=s_{(m,n),(m+1,n)}$ 
will be defined on horizontal edges from vertex 
$i=(m,n) \in \mathbb{Z}^2$ to vertex $j=(m+1,n) \in 
\mathbb{Z}^2$ as the representative 
(for a sphere) that is common to both the 
null planes $\Lambda_i$ and $\Lambda_j$, 
and $s_2=s_{(m,n),(m,n+1)}$ is defined analogously on vertical edges from 
$i$ to $\ell=(m,n+1)$.  We then define the principal curvatures by (\cite{BPW-MathAnn}, \cite{BHR3})
\begin{equation}\label{kappadefn} \kappa_{ij}=\frac{(s_{1},q)}{(s_{1},p)} \; , \;\;\; 
\kappa_{i\ell}=\frac{(s_{2},q)}{(s_{2},p)} \; . \end{equation}

As the principal curvature spheres $s_\alpha$ 
and principal curvatures $\kappa_{\alpha\beta}$ 
are defined on edges, not vertices, we lose the ability 
to look for points in the domain where $s_\alpha$ is 
exactly perpendicular to $p$ or $q$.  Thus we 
reformulate the conditions for singularities and 
parabolic or flat points by finding vertices in the domain 
at which the $\kappa_{\alpha \beta}$ change sign 
in at least one direction:  

\begin{defn}\label{defnOfSingVert}
For a $\Lambda$ as in Definition \ref{LegLiftDefn}, 
together with a choice of spaceform determined from a choice of $p$ and $q$, we say that $(m,n)$ is a {\em flat-or-parabolic-or-singular (FPS) vertex} if 
\[ \kappa_{(m-1,n),(m,n)} \cdot \kappa_{(m,n),(m+1,n)} \leq 0  \;\;\; 
\text{or} \;\;\;  
\kappa_{(m,n-1),(m,n)} \cdot \kappa_{(m,n),(m,n+1)} \leq 0 \; . \] 
\end{defn}

\begin{figure}
\begin{center}
\includegraphics[width=150mm]{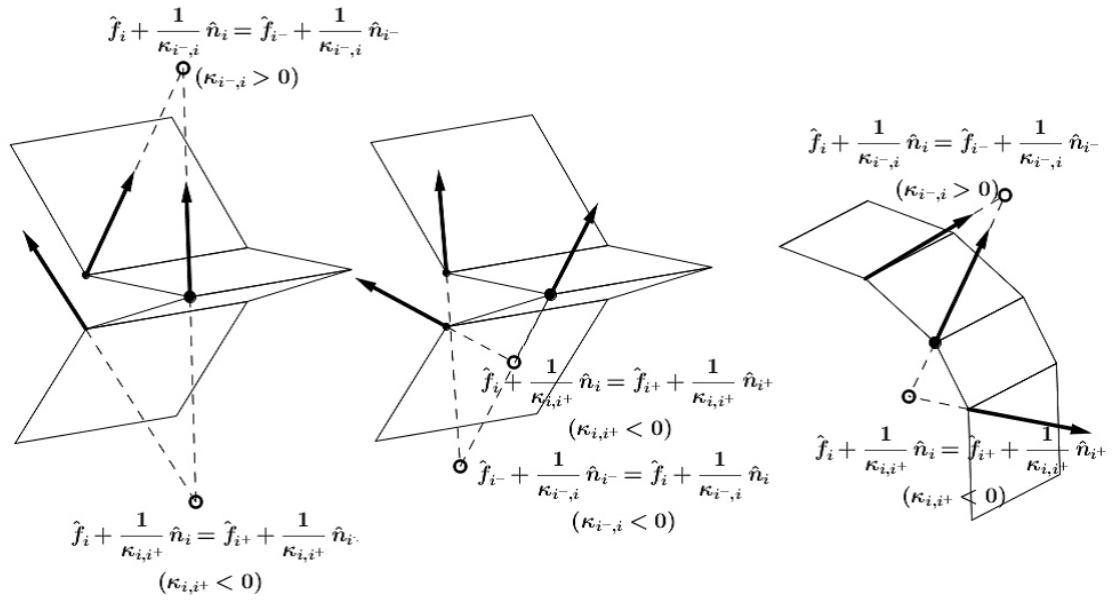}
\end{center}
\caption{Examples of FPS vertices on the left and right, and a non-example in the middle. The figure on the left shows a situation we should regard as a singularity, and the figure on the right shows what should be regarded as a flat or parabolic point.  The figure in the middle is neither.  Here, if $i=(m,n)$, 
then we have either $i^+=(m+1,n)$, $i^-=(m-1,n)$ or $i^+=(m,n+1)$, $i^-=(m,n-1)$.}
\end{figure}

When both $p$ and $q$ are non-null, switching $p$ and 
$q$ will result in the projected surface $f$ changing to its Gauss map $n$. In the smooth case, generically, a 
parabolic or flat point on one of the two surfaces corresponds to a singular point on the other, thus it is not surprising that these notions appear together in Definition \ref{defnOfSingVert}.  

In certain special cases, we can distinguish 
the singular points from the parabolic or flat points, which we will see here.  

As another approach for considering singularities on discrete surfaces, motivated by the second author's 
work \cite{Yashi2014}, we can define singular faces.  We come back to this in Definition \ref{defnOfSingFace},   
and examine criteria for singular faces, and also their 
relationships with singular vertices in some special cases.

\section{Smooth linear Weingarten surfaces of Bryant and Bianchi type in $\mathbb{H}^3$ and $\mathbb{S}^{2,1}$}\label{BrBi}

We include this section to motivate the discretizations in \S \ref{secfive}.
In $\mathbb{R}^{3,1}$ with signature $(-,+,+,+)$, with points $(x_0,x_1,x_2,x_3) \in \mathbb{R}^{3,1}$ described in matrix form as 
\[
X = 
\begin{pmatrix}
x_0 + x_3 & x_1 - i x_2 \\ x_1+ i x_2 & x_0 - x_3
\end{pmatrix} \; , 
\]
the metric is 
$\displaystyle \langle X, Y \rangle = \frac{-1}{2} \mathrm{tr} \left( 
X \begin{pmatrix} 0 & -i \\ i & 0 \end{pmatrix} Y^t \begin{pmatrix} 0 & -i \\ i & 0 \end{pmatrix} \right) $. We define 
\begin{eqnarray*}
&&\mathbb{H}^3 := \{ X \in \mathbb{R}^{3,1} | \langle X,X \rangle =-1 \} =\{ \pm F \overline{F}^t | F \in \mathrm{SL}_2 \mathbb{C}  \} , \\
&&\mathbb{S}^{2,1} := \{ X \in \mathbb{R}^{3,1} | \langle X,X \rangle =1 \} =\{ F \begin{pmatrix} 1 & 0 \\ 0 & -1 \end{pmatrix} \overline{F}^t | F \in \mathrm{SL}_2 \mathbb{C}  \} .
\end{eqnarray*}

We call a surface $\hat{f}$ in $\mathbb{H}^3$ a  {\it linear Weingarten surface of Bryant type} ({\it BrLW surface}, for short) if $\hat{f}$ satisfies
\begin{equation}\label{BrLWeqn}
2 t (H_{\hat f}-1) + (1-t) (K_{ext,\hat f}-1) = 0  , 
\end{equation}
where $K_{ext,\hat{f}}$ and $H_{\hat{f}}$ are the extrinsic Gaussian and mean curvatures of $\hat{f}$ with respect to $\mathbb{H}^3$, 
and we call a surface $\hat{n}$ in $\mathbb{S}^{2,1}$ a  {\it linear Weingarten surface of Bianchi type} ({\it BiLW surface}, for short) if $\hat{n}$ satisfies
\begin{equation}\label{BiLWeqn}
2 t (H_{\hat n}-1) - (1+t) (K_{ext,\hat n}-1) = 0 . 
\end{equation}

Solving 
\[
dE = E \begin{pmatrix}
0 & g^\prime \\ (g^\prime)^{-1} & 0 
\end{pmatrix} dz
\]
for $E \in \text{SL}_2(\mathbb{C})$, where $g$ is a holomorphic function 
with nonzero derivative $g^\prime = \partial_z g$ on a Riemann 
surface $M^2$ with local coordinate $z$, we take, for any 
constant $t \in \mathbb{R}$,  

\begin{equation}
\left\{
\begin{split}
L = \begin{pmatrix}
0 & \sqrt{ \mathcal{T} } \\ 
\frac{-1}{ \sqrt{ \mathcal{T} } } & 
\frac{-t \bar g}{ \sqrt{ \mathcal{T} } }
\end{pmatrix} \quad (\mathcal{T}:=1+t g \bar{g}), \\ 
\hat f =\mathrm{sgn(\mathcal{T})} E L \overline{E L}^t, \ \  
\hat n = \mathrm{sgn(\mathcal{T})} E L \begin{pmatrix}
1 & 0  \\ 0 & -1 \end{pmatrix}
\overline{E L}^t , 
\end{split}
\right.
\end{equation}  
making the genericity assumption $\mathcal{T} \neq 0$. 

Then $\hat{f}$ is a BrLW surface in $\mathbb{H}^3$ with unit normal vector field $\hat{n}$, since $\langle \hat{f},\hat{n} \rangle = \langle d\hat{f},\hat{n} \rangle = 0$. Moreover, $\hat{n}$ is a BiLW surface in $\mathbb{S}^{2,1}$. Here we outline a proof of this.

The three fundamental forms of $\hat f$ become, with $h := |g^\prime|^{-2} {\mathcal{T}}^{-2}$, 
\begin{eqnarray*} 
&&I = 
h \left\{ ((1-t) |g^\prime|^2 + {\mathcal{T}}^2)^2 
dx^2 + ((1-t) |g^\prime|^2 - {\mathcal{T}}^2)^2 
dy^2 \right\} \; ,\\
&&II = 
- h \left\{ (|g^\prime|^4 - (t |g^\prime|^2 - 
{\mathcal{T}}^2)^2) dx^2 + (|g^\prime|^4 - 
(t |g^\prime|^2 + 
{\mathcal{T}}^2)^2) dy^2 \right\} \; , \\
&&III = 
h \left\{ ((1+t) |g^\prime|^2 - {\mathcal{T}}^2)^2 
dx^2 + ((1+t) |g^\prime|^2 + {\mathcal{T}}^2)^2 
dy^2 \right\} \; . 
\end{eqnarray*}
The principal curvatures of $\hat f$ and $\hat n$ are then 
\begin{eqnarray*}
&&k_{1,\hat f} = - \frac{(1+t) |g^\prime |^2
-\mathcal{T}^2 }{(1-t)|g^\prime |^2+ \mathcal{T}^2 } \ ,\   
k_{2,\hat f} = - \frac{(1+t) |g^\prime |^2+ \mathcal{T}^2 }{(1-t) |g^\prime |^2-\mathcal{T}^2 } , \\
&&k_{1,\hat n} = \frac{1}{k_{1,\hat f} } , \ 
k_{2,\hat n} = \frac{1}{k_{2,\hat f} }  , \ \ H_{\hat f} = \frac{H_{\hat n}}{K_{ext,\hat n}}\ , \ \    
K_{ext,\hat f} = \frac{1}{K_{ext,\hat n}} \ ,
\end{eqnarray*} 
and so $\hat{f}$ satisfies Equation \eqref{BrLWeqn} and $\hat n$ satisfies Equation \eqref{BiLWeqn}.
In fact, all BrLW and BiLW surfaces without umbilics ($g^\prime$ would be 
zero at umbilics) can be constructed this way, using 
holomorphic functions $g$.  

Thus sufficient conditions for $\hat f$ and $\hat n$, respectively, to have 
singularities are 
\[ \mathcal{T}^4 = (1-t)^2 |g^\prime|^4 \; , \;\;\; 
\mathcal{T}^4 = (1+t)^2 |g^\prime|^4 \; , \]
respectively.  For certain special values of $t$ these conditions simplify 
as follows:
\begin{eqnarray*}
&&\hat f \;\; \text{with} \;\; t=0 : \;\;\;\; |g^\prime|=1, \\
&&\hat n \;\; \text{with} \;\; t=0 : \;\;\;\; |g^\prime|=1, \\
&&\hat f \;\; \text{with} \;\; t=1 : \;\;\;\; \text{null condition}, \\
&&\hat n \;\; \text{with} \;\; t=-1 : \;\;\;\; |g|=1 . 
\end{eqnarray*}

Because $\hat{f}$ and $\hat{n}$ are smooth well-defined maps that can have singularities, it is natural to lift to Lie sphere geometry 
in $\mathbb{R}^{4,2}$, with
\begin{equation}\label{extra2}
f =(\hat f,1,0)^t \; , \;\;\; n=(\hat n,0,1)^t 
\end{equation}
determined by 
\[ p=(0,0,0,0,0,1)^t \; , \;\;\; q=(0,0,0,0,-1,0)^t \; . \]  For a BrLW 
surface $\hat f \in \mathbb{H}^3=M_{p,q}^3$ with BiLW normal 
$\hat n \in \mathbb{S}^{2,1}=M_{q,p}^3$, we can define the Legendre lift 
$\Lambda = \text{span}\{ s_+,s_- \}$ for 
\[ s_\pm = b_\pm f+n \;\;\; 
\text{with} \;\;\; b_+=1 \;\;\; \text{and} \;\;\; b_-=\frac{t+1}{t-1} \; , \]
and then $s_\pm$ have constant conserved quantities 
\[ q_+=(0,0,0,0,1,1)^t \; , \;\;\; q_-=(0,0,0,0,t-1,t+1)^t \; . \]
in the sense that $(s_\pm,q_\pm)=0$, equivalently the equations 
$\Gamma^\pm q_\pm = 0$ for the associated families of flat connections 
hold (see \cite{BHR2}, \cite{wisky-next}).  
Furthermore, because $b_\pm$ are constant and because the elements $g_{ij}$ of the first fundamental form of $\hat{f}$ satisfy (using Equation \eqref{BrLWeqn}) 
\[
\pm \frac{\sqrt{g_{11}}}{\sqrt{g_{22}}} = \frac{1-\kappa_2}{1-\kappa_1} 
= \frac{-t-1+ (t-1) \kappa_2}{t+1- (t-1) \kappa_1} \ , 
\]
all of Equations (4.5) and (4.10) and (4.11) in \cite{wisky-next} hold, and 
so $s_\pm$ are isothermic sphere congruences.  Thus $\Lambda$ 
is an $\Omega$ surface with a pair of constant conserved quantities.  

Conversely, if we start with an $\Omega$ surface with constant 
conserved quantities $q_\pm$ for isothermic sphere congruences 
$s_\pm = b_\pm f+n$ respectively, we can reverse the above arguments 
to see that we obtain a BrLW surface $\hat f$ 
with BiLW normal $\hat n$ in the spaceforms $M_{p,q}^3$ and $M_{q,p}^3$, 
with $p$ and $q$ as above.  

This proves the next lemma, which was already understood in \cite{BHR2}:

\begin{lemma}[\cite{BHR2}]\label{BRandBiLemma}
All smooth BrLW and BiLW surfaces in $\mathbb{H}^3$ and $\mathbb{S}^{2,1}$ are projections of $\Omega$ 
surfaces with constant conserved quantities, at least one of which is lightlike. Conversely, for any smooth $\Omega$ surface with constant conserved quantities\footnote{We assume $q_+$, $q_-$ are not parallel, and that $\text{span} \{ q_{\pm} \}$ is not a null plane.} $q_\pm$, at least one of which is lightlike, its projections $\hat f$ and $\hat n$ given by choosing $p,q \in \text{span} \{ q_\pm \}$ are BrLW and BiLW surfaces, respectively.  
\end{lemma}

The same result holds for general linear Weingarten $\hat f$ and $\hat n$, even without the condition that at least one of the $q_\pm$ is lightlike, again see \cite{BHR2}.  However, here we consider only the cases given in Lemma 
\ref{BRandBiLemma}.  

\section{Discrete surfaces with Weierstrass-type representations}\label{secfive}

First we give Weierstrass-type representations for discrete surfaces using the more symmetric form of the base equation as in \S 6 of \cite{HRSY}.  

Let $g:\mathbb{Z}^2 \rightarrow \mathbb{C}$ be a function satisfying
\[
cr (g_i,g_j,g_k,g_{\ell}):=\frac{(g_i-g_j)(g_k-g_{\ell})}{(g_j-g_k)(g_{\ell}-g_i)}=\frac{ \alpha_{ij} }{ \alpha_{i \ell} } <0 \ ,
\]
where $\alpha_{ij}$ (resp. $\alpha_{i\ell}$) is a scalar function defined on the horizontal edges (resp. vertical edges) and unchanging with respect to vertical (resp. horizontal) shifts. A complex-valued function $g$ satisfying the above condition is called a {\em discrete holomorphic function} and $\alpha_{ij}, \ \alpha_{i\ell}$ are called {\em cross ratio factorizing functions}. Now we assume the discrete analog of $g^\prime \neq 0$, i.e. $dg_{ij} := g_j-g_i \neq 0$ and $dg_{i\ell} \neq 0$ for all quadrilaterals.  We again make the genericity assumption 
\[
\mathcal{T}_i:=1+t g_i \overline{g_i} \neq 0 
\]
for all vertices $i$, for the chosen 
constant $t \in \mathbb{R}$.  Take $\lambda \in 
\mathbb{R}$ to be any non-zero constant so that $1-\lambda 
\alpha_{ij} \neq 0$ on all edges.  Solving 
\[
E_i^{-1} E_j = \frac{1}{\sqrt{1-\lambda \alpha_{ij}}} \begin{pmatrix}
1 & dg_{ij} \\ \frac{\lambda \alpha_{ij}}{dg_{ij}} & 1 
\end{pmatrix} 
\]
and the analogous equation with $j$ replaced by $\ell$, 
for $E_i \in \text{SL}_2\mathbb{C}$ for all $i$, and defining 
\[ 
L_i = \begin{pmatrix}
0 & \sqrt{ \mathcal{T}_i } \\ 
\frac{-1}{ \sqrt{ \mathcal{T}_i } } & 
\frac{-t \bar g_i}{ \sqrt{ \mathcal{T}_i } }
\end{pmatrix} , 
\]
the surface $\hat f$ and its normal $\hat n$ 
\[ 
\hat f_i =\mathrm{sgn}(\mathcal{T}_i) E_i L_i \overline{E_i L_i}^t , \quad 
\hat n_i =\mathrm{sgn(\mathcal{T}_i)} E_i L_i \begin{pmatrix}
1 & 0 \\ 0 & -1 
\end{pmatrix}
\overline{E_i L_i}^t ,
\]
we will see that these are discrete BrLW surfaces and BiLW surfaces in $\mathbb{H}^3$ and 
$\mathbb{S}^{2,1}$, respectively.  Direct computations confirm the following 
lemma: 

\begin{lemma}\label{princurv}
For any choice of $t$, we have the following: 
\begin{itemize}
\item $d\hat f_{ij} || d\hat n_{ij}$, $d\hat f_{i\ell} || d\hat n_{i\ell}$ in 
$\mathbb{R}^{3,1}$ for all edges $ij$, $i\ell$, and the principal curvatures $\kappa_{i*}$ satisfy 
\[
dn_{i*} = -\kappa_{i*} df_{i*} \ , 
\]
and furthermore 
\begin{equation}\label{lemmaEqn:aboutkappa}
\kappa_{i*} = \frac{-|dg_{i*}|^2 (1+t)+ 
(1+t|g_i|^2) (1+t|g_*|^2) \lambda \alpha_{i*}}{-|dg_{i*}|^2 (-1+t)+ 
(1+t|g_i|^2) (1+t|g_*|^2) \lambda \alpha_{i*}} \; , 
\end{equation}
for $*=j,\ell$. 

\item $1+t |g_i|^2 > 0$, resp. $1+t |g_i|^2 < 0$, if and only if $\hat f_i$ lies in the upper, resp. lower, sheet of $\mathbb{H}^3$.  

\item $\hat f_i$, $\hat f_j$, $\hat f_k$, $\hat f_\ell$ lie 
in a plane in $\mathbb{R}^{3,1}$, 
and thus are concircular in $\mathbb{H}^3$.  
\end{itemize}
\end{lemma}

\begin{corollary}\label{forrefingthisintmpsection}
For any choice of $t$, the parallel surfaces 
\[ \cosh \theta \cdot \hat f + \sinh \theta \cdot 
\hat n \; , \;\;\; \cosh \theta \cdot \hat n + \sinh \theta \cdot 
\hat f \] are concircular for all $\theta \in \mathbb{R}$.  
\end{corollary}

\begin{proof}
$d\hat f_{i*} || d\hat n_{i*}$ and the fact that corresponding quadrilaterals of  $f$ and $n$ lie in parallel planes imply that corresponding quadrilaterals of $\cosh \theta \cdot \hat{f} + \sinh \theta \cdot \hat{n}$ also lie in parallel planes, proving the corollary.  
\end{proof}

Like in Equation \eqref{extra2}, we can lift $\hat f$ and $\hat n$ to 
$f,n \in \mathbb{R}^{4,2}$, producing a discrete Legendre immersion 
$\Lambda=\text{span}\{ f,n \}$.  We define 
\[ \mathcal{A}(f,f)_{ijk\ell} := \frac{1}{2} df_{ik} \wedge df_{j\ell} \; , \]
and we can define real-valued functions $H=H_{ \hat{f} }$ and 
$K=K_{ \hat{f} }$ on faces by 
\begin{equation}\label{defnOfH} 
\mathcal{A}(f,n)_{ijk\ell} := 
\frac{1}{4} \{ df_{ik} \wedge dn_{j\ell} + 
dn_{ik} \wedge df_{j\ell} \} = 
-H \cdot \mathcal{A}(f,f)_{ijk\ell} \; , 
\end{equation}
\begin{equation}\label{defnOfK} 
\mathcal{A}(n,n)_{ijk\ell} = \frac{1}{2} 
dn_{ik} \wedge dn_{j\ell} 
= K \cdot 
\mathcal{A}(f,f)_{ijk\ell} \; . \end{equation}
We have the following definition: 

\begin{defn}[\cite{BHR3}] \label{GaussMeanDefn} 
We call $K_{\hat f}$ and $H_{\hat f}$ 
the (extrinsic) {\em Gaussian} and {\em mean curvature}
of the projection $\hat f$ of $\Lambda$ to the spaceform given by $p,q$.  
\end{defn}

Proven similarly to the corresponding result for $\mathbb{R}^3$ in 
\cite{BPW-MathAnn}, using item 1 of Lemma \ref{princurv}, we have: 

\begin{lemma}\label{GMbyPrinp}
For all choices of spaceforms, we have 
\begin{eqnarray*}
&&H_{\hat f} = \frac{ \kappa_{ij}\kappa_{k\ell}-\kappa_{i\ell}\kappa_{jk}
 }{\kappa_{ij}-\kappa_{i\ell}-\kappa_{jk}+\kappa_{k\ell}} \ , \\  
&&K_{\hat f} = \frac{ \kappa_{ij} \kappa_{jk} \kappa_{k\ell} \kappa_{i\ell} }{\kappa_{ij}-\kappa_{i\ell}-\kappa_{jk}+
\kappa_{k\ell}}  \left( -\frac{1}{\kappa_{ij} } +\frac{1}{\kappa_{jk} }+\frac{1}{\kappa_{i\ell} }-\frac{1}{\kappa_{k\ell} } \right) \ . 
\end{eqnarray*}
\end{lemma}

\begin{proof}

The compatibility condition $\hat{n}_{ij}+\hat{n}_{jk}=\hat{n}_{i\ell} +\hat{n}_{\ell k}$ for $\hat{n}$ implies
\begin{eqnarray*}
\kappa_{ij}d \hat{f}_{ij}+\kappa_{jk}d \hat{f}_{jk}=\kappa_{i\ell}d \hat{f}_{i\ell}+\kappa_{k\ell}d \hat{f}_{\ell k} \ .
\end{eqnarray*}
Thus
\begin{eqnarray*}
&&\kappa_{ij}d \hat{f}_{ij}+\kappa_{jk} (d \hat{f}_{ik} -d \hat{f}_{ij} )=\kappa_{i\ell}d \hat{f}_{i\ell}+\kappa_{k\ell}(d \hat{f}_{ik} -d \hat{f}_{i\ell} ) \\
&&\Rightarrow d \hat{f}_{ik} =c_1 d \hat{f}_{ij}+c_2 d \hat{f}_{i\ell} \ ,
\end{eqnarray*}
where 
\begin{equation}\label{c1c2}
c_1=\frac{-\kappa_{ij} +\kappa_{jk} }{ \kappa_{jk}-\kappa_{k\ell} }, \ \ c_2=\frac{-\kappa_{lk} +\kappa_{i\ell} }{ \kappa_{jk}-\kappa_{k\ell} } .
\end{equation}
Similarly, by the compatibility condition for $\hat{f}$ and the condition $d n_{**}=-\kappa_{**} d f_{**}$, we have $d \hat{n}_{ik}=c_3 d f_{ij}+c_4 d f_{il} \ $, 
where
\begin{equation}\label{c3c4}
c_3=\frac{\kappa_{k\ell} (\kappa_{jk} -\kappa_{ij}) }{ \kappa_{k\ell}-\kappa_{jk} } \ , \ \ c_4=\frac{\kappa_{jk} (\kappa_{i\ell} -\kappa_{k\ell}) }{ \kappa_{k\ell}-\kappa_{jk} }  \ .
\end{equation}
Note that $d \hat{f}_{j\ell}=d \hat{f}_{i\ell}-d \hat{f}_{ij}, \ d \hat{n}_{j\ell}=d \hat{n}_{i\ell}-d \hat{n}_{ij}$, and we have
\begin{equation}\label{c1c2c3c4}
H_{\hat{f} }=\frac{\kappa_{i\ell} c_1 +\kappa_{ij} c_2-c_3 -c_4}{2(c_1+c_2)} \ , \ \  K_{\hat{f} }=-\frac{\kappa_{i\ell} c_3+\kappa_{ij} c_4}{c_1+c_2} \ .
\end{equation}
Substituting Equations \eqref{c1c2} and \eqref{c3c4} into Equation \eqref{c1c2c3c4}, we have $H_{\hat{f} }$ and $K_{\hat{f} }$ as in Lemma \ref{GMbyPrinp}.
\end{proof}

We can similarly define the Gaussian and mean curvatures $K_{\hat{n}}$, $H_{\hat{n}}$ of $\hat{n}$, and we see that 
\begin{equation}\label{extra3}
K_{\hat{n}}=\frac{1}{K_{\hat{f}} }, \ \ H_{\hat{n}}=\frac{H_{\hat{f}} }{K_{\hat{f}} }, \ \ \kappa_{ij,\hat{n} }=\frac{1}{\kappa_{ij,\hat{f} }}, \ \ \kappa_{il,\hat{n} }=\frac{1}{\kappa_{il,\hat{f} }}. 
\end{equation}
One can confirm the next lemma via Lemma \ref{GMbyPrinp} and Equations \eqref{lemmaEqn:aboutkappa}, \eqref{extra3}:

\begin{lemma}\label{extra4}
The mean and Gaussian curvatures $H_{\hat f}$ and $K_{\hat f}$ of a discrete  surface $\hat{f}$ with Weierstrass-type representation in $\mathbb{H}^3$ satisfy 
\begin{equation}\label{extra5}
2 t (H_{\hat f}-1)+(1-t) (K_{\hat f}-1)=0 \ , 
\end{equation}
and the mean and Gaussian curvatures $H_{\hat n}$ and $K_{\hat n}$ of a discrete surface $\hat n$ with Weierstrass-type representation in $\mathbb{S}^{2,1}$ satisfy 
\begin{equation}\label{extra7}
2 t (H_{\hat n}-1)-(1+t) (K_{\hat n}-1)=0 \ . 
\end{equation}
\end{lemma}

\begin{figure}
\begin{center}
\includegraphics[width=150mm]{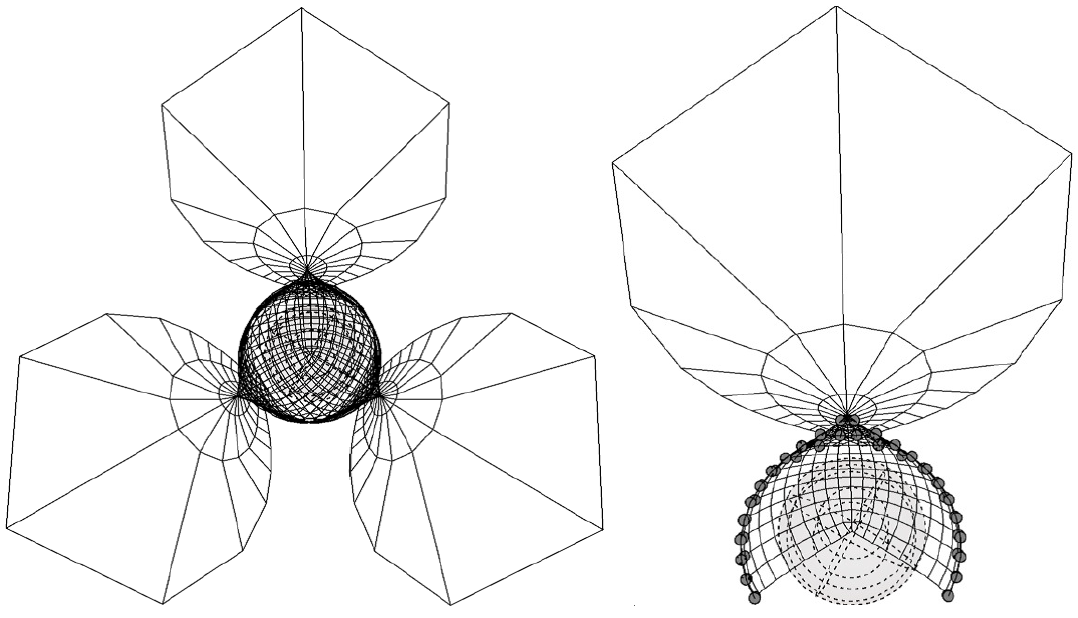}
\end{center}
\caption{One example of a discrete BrLW surface in $\mathbb{H}^3$ with $t=-2$. Red vertices are the FPS vertices of the surface. In order to draw the surface, we use stereographic projection from the south pole $(0,0,0,-1)$. When $ 1+t |g_i|^2 >0 $, the surface is projected to the inside of the unit ball $\mathbb{B}^3$, whose boundary is drawn in gray above. When $ 1+t |g_i|^2 <0 $, it is projected to the outside of $\mathbb{B}^3$. One-third of the surface is shown on the right.}
\end{figure}

Thus we know that the discrete surfaces with Weierstrass-type representations defined here are included amongst the discrete BrLW and BiLW  surfaces defined in \cite{BHR3}, by the following Proposition \ref{ConsQuanBrBi} from \cite{BHR3}. This proposition also includes discrete minimal  surfaces in $\mathbb{R}^3$ and their parallel surfaces in $\mathbb{R}^3$, as well as parallel surfaces of discrete maximal surfaces in $\mathbb{R}^{2,1}$.

\begin{proposition}[\cite{BHR3}] \label{ConsQuanBrBi}
All discrete BrLW and BiLW surfaces in $\mathbb{H}^3$ and $\mathbb{S}^{2,1}$, and all parallel surfaces of discrete minimal surfaces in $\mathbb{R}^3$ and discrete maximal surfaces in $\mathbb{R}^{2,1}$, are projections of discrete $\Omega$ surfaces with constant conserved quantities,  at least one of which is lightlike.  Conversely, for any discrete $\Omega$ surface with constant 
conserved quantities\footnote{Again we assume $q_+$, $q_-$ are not parallel, and that $\text{span} \{ q_{\pm} \}$ is not a null plane.} $q_\pm$, at least one of which is lightlike, its projections $\hat f$ and $\hat n$ given by choosing $p,q \in \mathrm{span} \{ q_\pm \}$ are discrete BrLW and BiLW surfaces, respectively, or $\hat{f}$ is either a parallel surface of a discrete minimal surface in $\mathbb{R}^3$ or maximal surface in $\mathbb{R}^{2,1}$.
\end{proposition}

In the smooth case, as mentioned in \cite{KU}, parallel surfaces of BrLW surfaces in $\mathbb{H}^3$ are also BrLW surfaces, and BrLW surfaces are classified into the following three types:
\begin{enumerate}
\item flat surfaces (BrLW surfaces with $t=0$),
\item linear Weingarten surfaces of hyperbolic type (BrLW surfaces with $t>0$),
\item linear Weingarten surfaces of de Sitter type (BrLW surfaces with $t<0$).
\end{enumerate}
Parallel surfaces of each type belong to the same type. Thus parallel surfaces of a flat front are also flat. Likewise, parallel surfaces of the other two types again belong to the same types.

Here we see that the same result as in the smooth case holds also in the discrete case. Let $\hat{f}$ be a discrete BrLW surface in $\mathbb{H}^3$.  From \cite{OY}, we have that the Gaussian and mean curvatures $K^{\theta}_{\hat{f} }$, $H^{\theta}_{\hat{f} }$ of the parallel surface $\hat{f}^{\theta }$ at oriented distance $\theta$ are
\begin{eqnarray*}
&&K^{\theta}_{\hat{f} }=\frac{K_{\hat{f} } \cosh^2 \theta -H_{\hat{f} } \sinh 2 \theta +\sinh ^2 \theta }{\cosh^2 \theta -H_{\hat{f} } \sinh 2 \theta +K_{\hat{f} } \sinh ^2 \theta } \ , \\ 
&&H^{\theta}_{\hat{f} } = -\frac{(K_{\hat{f} }+1)\sinh (2\theta)-2H_{\hat{f} } \cosh (2\theta)}{2\{ \cosh^2 \theta -H_{\hat{f} } \sinh (2\theta) +K_{\hat{f} } \sinh^2 \theta \} } \ .
\end{eqnarray*}
Observing that $\hat{f}=(\hat{f}^{\theta })^{-\theta }$, we have
\begin{equation}\label{extra6}
\left\{
\begin{split}
K_{\hat{f} }=\frac{K^{\theta }_{\hat{f} } \cosh^2 \theta +H^{\theta }_{\hat{f} } \sinh 2 \theta +\sinh ^2 \theta }{\cosh^2 \theta +H^{\theta }_{\hat{f} } \sinh 2 \theta +K^{\theta }_{\hat{f} } \sinh ^2 \theta } \ , \\ 
H_{\hat{f} } = \frac{(K^{\theta }_{\hat{f} }+1)\sinh (2\theta)+2H^{\theta }_{\hat{f} } \cosh (2\theta)}{2\{ \cosh^2 \theta +H^{\theta }_{\hat{f} } \sinh (2\theta) +K^{\theta }_{\hat{f} } \sinh^2 \theta \} } \ .
\end{split}
\right.
\end{equation}
Substituting Equation \eqref{extra6} into Equation \eqref{extra5}, we have
\[
2T(H^{\theta }_{\hat{f} }-1)+(1-T)(K^{\theta }_{\hat{f} }-1)=0 ,
\]
where $T=\mathrm{e}^{-2\theta }t$. Thus discrete BrLW surfaces in $\mathbb{H}^3$ are classified into the three types $(1) - (3)$ mentioned above.  Similarly, discrete BiLW surfaces in $\mathbb{S}^{2,1}$ are classified into three types.

\section{Singular vertices on discrete nonzero CGC surfaces in $M^3$}\label{SingVertCGC}

When a smooth surface has CGC $K=\kappa_1 \kappa_2 \neq 0$, then when one of the $\kappa_\alpha$ passes through zero, the other passes through 
infinity, and we can always call this a singular point.  This is precisely what allowed for the description of singularities of discrete flat (i.e. $K \equiv 1$) surfaces in $\mathbb{H}^3$ as given in \cite{HRSY}. Here we develop that into a definition without reliance on a Weierstrass type representation, extending it to all discrete surfaces in any $M^3$ with nonzero constant Gaussian curvature.

\begin{defn}\label{defn:CGCcase}
Consider $\Lambda$ as in Definition \ref{LegLiftDefn}, together with a choice of spaceform determined by choosing $p$ and $q$, that has projection $\hat f$ with nonzero constant discrete Gaussian curvature $K_{\hat f}$.  We say that $(m,n)$ is a {\em singular vertex} of $\hat f$ if 
\[ 
\kappa_{(m-1,n),(m,n)} \cdot \kappa_{(m,n),(m+1,n)} \leq 0  \;\;\; 
\text{or} \;\;\;  
\kappa_{(m,n-1),(m,n)} \cdot \kappa_{(m,n),(m,n+1)} \leq 0 \; . 
\]
\end{defn}

For a $K \equiv 1$ surface with Weierstrass-type representation in $\mathbb{H}^3$, it was shown in \cite{HRSY} that, without loss of generality, $ | \kappa _{(m,n), (m,n+1)} |  >1$ and $ | \kappa _{(m,n), (m+1,n)} |  <1$ for all $m$ and  $n$, which we note in the following theorem:

\begin{theorem}\label{SingVertFlat}
In the case of a $K \equiv 1$ surface in $\mathbb{H}^3$ with Weierstrass-type representation so that the horizontal edges have principal curvatures with absolute value greater than $1$, the first inequality in Definition  \ref{defn:CGCcase} is equivalent to the definition of singular 
vertices for discrete flat surfaces in $\mathbb{H}^3$ as given 
in \cite{HRSY}.
\end{theorem}

\begin{proof}
By Lemma \ref{princurv}, for $t=0$ we have 
\[ \kappa_{i*} = \frac{-|dg_{i*}|^2 + \lambda \alpha_{i*}}{|dg_{i*}|^2 + 
\lambda \alpha_{i*}} \: . \]  
Let $p_-$, $p$ and $p_+$ be three consecutive vertices in one 
direction in the lattice domain. We can define singularities on discrete flat (i.e. $K \equiv 1$) surfaces in $\mathbb{H}^3$, now without referring to caustics 
as in \cite{HRSY}, by 
simply using the condition 
\[
\frac{-|dg_{p_-p}|^2 + \lambda \alpha_{p_-p}}{|dg_{p_-p}|^2 + 
\lambda \alpha_{p_-p}} \cdot 
\frac{-|dg_{pp_+}|^2 + \lambda \alpha_{pp_+}}{|dg_{pp_+}|^2 + 
\lambda \alpha_{pp_+}} < 0 , 
\]
as understood in \cite{HRSY}.  
\end{proof}

However, our definition allows the second inequality in Definition \ref{defn:CGCcase}, which allows us to include more singular vertices (see Figure \ref{flatinH3}).

\begin{figure}
\begin{center}
\includegraphics[height=85mm]{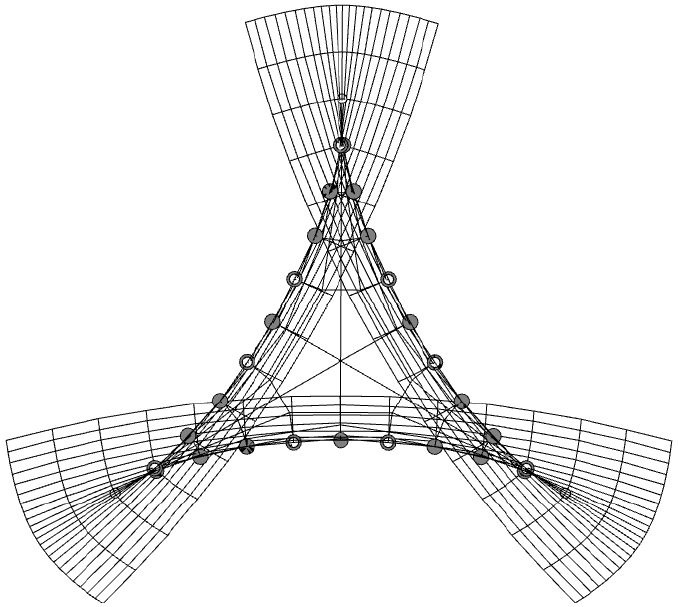}
\end{center}
\caption{A flat surface in $\mathbb{H}^3$ with its singular vertices in the sense of \cite{HRSY} shown with solid gray dots, and the extra singular vertices that would be included by Definition \ref{defn:CGCcase} shown with hollowed-out dots. }
\label{flatinH3}
\end{figure}

\section{Discrete minimal surfaces and their parallel surfaces}\label{secseven}

\subsection{Smooth minimal surfaces in $\mathbb{R}^3$} 
We can always take a smooth constant mean curvature 
(CMC) surface in a $3$-dimensional Riemannian spaceform 
to have local isothermic coordinates 
$z=u+iv$ on $M^2$, $u,v \in \mathbb{R}$ (away from umbilic points), 
and then the Hopf differential becomes $r dz^2$ for some real constant 
$r$.  Rescaling the coordinate $z$ by a constant real factor, 
we may assume $r=1$.  
So we now assume we have an isothermic minimal surface 
in $\mathbb{R}^3$ with Hopf differential $Q=dz^2$. Then 
\begin{equation}\label{eqn:HDthing} \frac{Q}{dg} = 
\frac{dz}{g^\prime} , 
\end{equation} 
where $g$ is the stereographic projection of the Gauss map to the complex plane, and $g^\prime = dg/dz$.  The map $g$ taking $z$ in the domain of the immersion 
(of the surface) to $\mathbb{C}$ is holomorphic.  
We avoid umbilics, so we have $g^\prime \neq 0$.  
We are concerned with only local behavior,  
so we can ignore the possibility that $g$ has poles.  
Then the Weierstrass representation is
\begin{equation}\label{minsurfeqn} 
\hat{f} = \mbox{Re} \int_{z_0}^z ( 1-g^2,i+ 
i g^2, 2g)^{t} \frac{dz}{g^\prime} \; , 
\end{equation}
with the last factor coming from \eqref{eqn:HDthing}.  
The metric of $\hat{f}$ is
\begin{equation}\label{sdm-eq3-mincase}
\frac{(1+|g|^2)^2}{|g^{\prime}|^2}dzd\bar{z} \; . 
\end{equation}
By direct computation, we have: 

\begin{lemma}\label{lem:7pt1} For a smooth minimal surface $\hat{f}$ as given in Equation \eqref{minsurfeqn}, the partial derivatives of $\hat{f}$ are 
\begin{eqnarray*} 
\hat{f}_u = {\rm Re} \left( \frac{1-g^2}{g_u}, \frac{i(1+g^2)}{g_u}, \frac{2 g}{g_u}  
\right)^{t} , \\
\hat{f}_v = -{\rm Re} \left( \frac{1-g^2}{g_v}, \frac{i(1+g^2)}{g_v}, \frac{2 g}{g_v} 
\right)^{t} . 
\end{eqnarray*}
Furthermore, the principal curvatures of the surface are 
\begin{equation}\label{eqn:kappaoneandtwo} 
\pm \kappa_1 = \mp \kappa_2 = \frac{2 |g^\prime|^2}{(1+|g|^2)^2} \; . 
\end{equation}
\end{lemma}

The next lemma will be used as motivation for the discussion about discrete minimal surfaces that follows it:

\begin{lemma}\label{smoothMSparsurfs}
Any parallel surface of a minimal surface in $\mathbb{R}^3$ without umbilics will have constant harmonic mean curvature, and will 
have neither parabolic nor flat points.  
\end{lemma}

\begin{proof}
For surfaces in $\mathbb{R}^{3}$ with Gaussian and mean curvatures $K$ and $H$, the parallel 
surfaces at distance $\rho$ have Gaussian and mean curvatures 
\[
\hat K = 
\frac{K}{1-2 \rho H + \rho^2 K} \; , \;\;\; 
\hat H = \frac{H-\rho K}{1-2 \rho H + \rho^2 K} \; , 
\]
so when we have a minimal surface ($H= 0$) in $\mathbb{R}^3$, 
\[
\hat K = \frac{K}{1 + \rho^2 K} \; , \;\;\; \hat H = \frac{-\rho K}{1 + 
\rho^2 K} \; . 
\]
If the minimal surface has no umbilics, then $K \neq 0$, which implies 
no parallel surface can have any parabolic or flat points.

The parallel surfaces all have constant harmonic mean curvature since 
$\frac{\hat H}{\hat K} = - \rho$. 
\end{proof}

\subsection{Discrete minimal surfaces in $\mathbb{R}^3$} 
Analogously to the smooth case, a suitable representation (or definition, see \cite{BobPink2}, \cite{Udo1}) for discrete minimal surfaces (equivalently, defined as discrete surfaces with $H_{\hat{f} } \equiv 0$) is, with $*=j,\ell$, 
\begin{equation}\label{dismin}
\hat{f}_* -\hat{f}_i = \frac{\alpha_{i*} }{2}{\rm Re} \left(
\frac{1-g_*g_i}{g_*-g_i}, \frac{ \sqrt{-1} (1+g_*g_i)}{g_*-g_i}, \frac{g_*+g_i}{g_*-g_i}  \right)^{t} , 
\end{equation}
where the map $g$ from a domain in $\mathbb{Z}^2$ to 
$\mathbb{C}$ is a discrete holomorphic function with cross ratio 
factorizing function $\alpha_{i*}$.  
As in the smooth case where we avoided umbilics, likewise here we assume 
\[ g_*-g_i \neq 0 \; . \]  

\begin{example}
The discrete holomorphic function $g=c (m+n\sqrt{-1})$ for $c$ a complex constant 
will produce a discrete minimal Enneper surface.  
The discrete holomorphic function $g= e^{c_1 m+ c_2 n \sqrt{-1} }$ for 
choices of real constants $c_1$ and $c_2$ so that the cross ratio is 
identically $-1$ will produce a discrete minimal catenoid. (See \cite{BobPink2} for graphics.)
\end{example}

Furthermore, the principal curvatures $\kappa_{i*}$ defined on edges 
(similarly to \eqref{eqn:kappaoneandtwo}) are 
\[
\kappa_{i*} = - \frac{4 |dg_{i*}|^2}{\alpha_{i*} (1+|g_i|^2)(1+|g_*|^2)} \; . 
\]

Based on Lemma \ref{smoothMSparsurfs}, we can justify the following 
definition: 

\begin{defn}\label{defn:minsurfparallelcase}
For any discrete minimal surface, we say 
that $(m,n)$ is a {\em singular vertex} of any given parallel surface if the principal curvatures $\kappa_{**}$ of that parallel surface satisfy
\begin{eqnarray*} 
\kappa_{(m-1,n),(m,n)} \cdot \kappa_{(m,n),(m+1,n)} \leq 0  \;\;\; 
\text{or} \\ 
\kappa_{(m,n-1),(m,n)} \cdot \kappa_{(m,n),(m,n+1)} \leq 0 \ .
\end{eqnarray*}
\end{defn}

\begin{figure}
\begin{center}
\includegraphics[width=150mm]{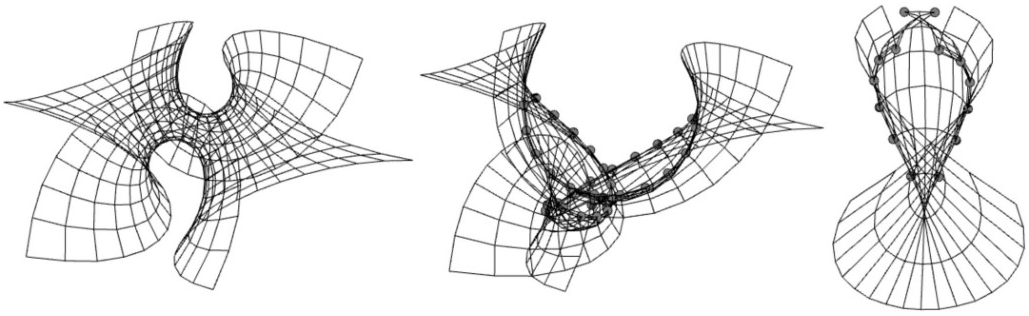}
\end{center}
\caption{A discrete higher-order Enneper minimal surface in $\mathbb{R}^3$, its parallel surface at distance $20$, and a one-third piece of the parallel surface, with singular vertices marked.}
\end{figure}

\section{Discrete maximal surfaces and their parallel surfaces}\label{seceight}

Here we give the analogous situation as in \S \ref{secseven}, but now in 
Lorentz $3-$space.  

\subsection{Smooth maximal surfaces in $\mathbb{R}^{2,1}$}\label{sdm-subsec1}

First we briefly review smooth maximal surfaces.  Let 
\[
\mathbb{R}^{2,1}:= (\{ (x_1,x_2,x_0)^t | x_j \in \mathbb{R}\} , \langle
\cdot ,\cdot \rangle )
\]
be 3-dimensional Minkowski space with the Lorentz metric signature 
$(+,+,-)$.  

Note that, for fixed $d \in \mathbb{R}$ and vector $n \in \mathbb{R}^{2,1}
\setminus \{ 0 \}$, a plane $\mathcal{P}=\{ x \in \mathbb{R}^{2,1} \mid
\langle x,n \rangle =d \}$ is {\em spacelike} or {\em timelike} or {\em
lightlike} when $n$ is timelike or spacelike or lightlike, respectively.
Furthermore, a smooth surface in $\mathbb{R}^{2,1}$ is called spacelike if its
tangent planes are spacelike.
Let 
\[ 
\hat{f} :M^2 \rightarrow \mathbb{R}^{2,1} 
\]
be a conformal immersion,
where $M^2$ is a simply-connected domain 
in $\mathbb{C}$ with complex coordinate $z=u+iv \ (u, v \in \mathbb{R})$.
$\hat{f}$ is a maximal surface if it is spacelike (which follows automatically
from the conformality condition) with mean curvature identically $0$.

Defining
\begin{eqnarray*}
&&\mathbb{H}^2_+:=\{ x=(x_1,x_2,x_0)^t \in \mathbb{R}^{2,1} | \langle x,x \rangle =-1,
\ x_0 >0 \}, \\
&&\mathbb{H}^2_-:=\{ x=(x_1,x_2,x_0)^t \in \mathbb{R}^{2,1} | \langle x,x \rangle =-1,
\ x_0 <0 \} ,
\end{eqnarray*}
we have the following statement, analogous to the case of smooth minimal
surfaces in $\mathbb{R}^3$, as in \eqref{minsurfeqn} (and 
having a similar proof): 
Away from umbilic points, smooth maximal surfaces lie in the class of
isothermic surfaces, and each such surface 
can be represented with isothermic coordinates $(u,v)$, $z=u+iv$, as 
\begin{equation}\label{sdm-eq2}
\hat{f}={\rm Re} \int \left(
1+g^2, i(1-g^2), -2g
\right)^{t} \frac{dz}{g^{\prime}}
\end{equation}
for some choice of smooth holomorphic function $g:M^2 \rightarrow 
\mathbb{C}$.  
The Gauss map of $\hat{f}$ lies in $\mathbb{H}^2_{+} \cup \mathbb{H}^2_{-}$, 
and its stereographic projection to $\mathbb{C}$ is $g$.  

Differentiating Equation \eqref{sdm-eq2} gives the following equations (analogous to Lemma \ref{lem:7pt1}):
\begin{eqnarray*}
\hat{f}_u = {\rm Re} \left( \frac{1+g^2}{g_u}, \frac{i(1-g^2)}{g_u},
-\frac{2g}{g_u} \right)^{t} , \\ 
\hat{f}_v = -{\rm Re} \left( \frac{1+g^2}{g_v}, \frac{i(1-g^2)}{g_v},
-\frac{2g}{g_v} \right)^{t}  . 
\end{eqnarray*}

\begin{remark}
Unlike the case of the Weierstrass representation for minimal surfaces in
$\mathbb{R}^3$, smooth maximal surfaces in $\mathbb{R}^{2,1}$ have
singularities when $|g|=1$, because the metrics
\begin{equation}\label{sdm-eq3}
\frac{(1-|g|^2)^2}{|g^{\prime}|^2}dzd\bar{z}
\end{equation}
of the smooth maximal surfaces can degenerate, due to the minus sign
in the numerator in Equation \eqref{sdm-eq3}, unlike the plus sign we 
have for the metrics of minimal surfaces in $\mathbb{R}^3$, as in 
\eqref{sdm-eq3-mincase}.  
\end{remark}

The principal curvatures of $\hat{f}$ are (analogous to 
\eqref{eqn:kappaoneandtwo}) 
\[ \pm \kappa_1 = \mp \kappa_2 = \frac{2 |g^\prime|^2}{(1-|g|^2)^2} \; . \]

By exactly the same proof as for Lemma \ref{smoothMSparsurfs}, we have: 

\begin{lemma}\label{smoothMSparsurfs-max}
Any parallel surface of a maximal surface in $\mathbb{R}^{2,1}$ without umbilics will have constant harmonic mean curvature, and will have neither parabolic nor flat points.  
\end{lemma}

\subsection{Discrete maximal surfaces in
$\mathbb{R}^{2,1}$}\label{sdm-subsec2}

The following theorem was proven in \cite{Yashi2014} (analogous to \eqref{dismin}): 

\begin{proposition}\label{dis-maximal-weierstrass}
Discrete maximal surfaces $\hat{f}$ (defined as discrete surfaces with $H_{\hat{f}} \equiv 0$ 
in $\mathbb{R}^{2,1}$), maps from $\mathbb{Z}^2$ (or 
some subdomain) to $\mathbb{R}^{2,1}$, can be constructed using discrete
holomorphic functions $g$ from the same domain to the
complex plane $\mathbb{C}$ by solving
\begin{equation}\label{dismax}
\hat{f}_* -\hat{f}_i = \frac{\alpha_{i*} }{2}{\rm Re} \left( \frac{1+g_*g_i}{g_*-g_i}, \frac{ \sqrt{-1} (1- g_*g_i)}{g_*-g_i} ,  -\frac{g_*+g_i}{g_*-g_i} \right)^{t} , 
\end{equation}
with $\alpha_{i*}$ the cross ratio factorizing functions for $g$.
Conversely, any discrete maximal surface satisfies \eqref{dismax} for
some discrete holomorphic function $g$.
\end{proposition}

\begin{lemma}
The principal curvatures $\kappa_{i*}$ of $\hat{f}$ defined on edges are 
\[ 
\kappa_{i*} = - \frac{4 |dg_{i*}|^2}{\alpha_{i*} (1-|g_i|^2)(1-|g_*|^2)} \; . 
\]
\end{lemma}

We recall the following definition of singular faces as in \cite{Yashi2014}: 

\begin{defn}\label{defnOfSingFace}
A face of $\hat{f}$ with vertices 
$\hat{f}_i,\hat{f}_j,\hat{f}_k,\hat{f}_\ell$ is {\em singular} if those four vertices lie 
in a non-spacelike plane.  
\end{defn}

It was proven in \cite{Yashi2014} that 
a quadrilateral of $\hat{f}$ is singular if and only if the 
corresponding circumcircle of $g$ intersects the unit circle $\mathbb{S}^1 \subset \mathbb{C}$. From this we can conclude the following theorem: 

\begin{theorem}\label{thm8pt4}
Let $p_-$, $p$ and $p_+$ be three consecutive vertices in one 
direction in the lattice domain of a 
maximal surface $\hat{f}$ in $\mathbb{R}^{2,1}$, 
with corresponding values $g_-$, 
$g$ and $g_+$ for the discrete holomorphic function in the 
Weierstrass type representation \eqref{dismax}.  Suppose $p$ is an FPS vertex.  
Then the pair of faces adjacent to the edge $p_-p$ are singular, or 
the pair of faces adjacent to the edge $pp_+$ are singular, 
including the possibility that all four faces are singular.  
\end{theorem}

\begin{proof}
Because
\[ 
\kappa_{p_-p}\kappa_{pp_+} = (\text{nonnegative term}) 
(1-|g_{p_-}|^2)(1-|g_{p_+}|^2) \ ,
\]
$\kappa_{p_-p}\kappa_{pp_+}<0$ implies 
at least one of $(1-|g_{p}|^2)(1-|g_{p_-}|^2)$ or 
$(1-|g_{p}|^2)(1-|g_{p_+}|^2)$ is negative, and so at least one of the edges 
$p_-p$ or $pp_+$ has two adjacent singular faces.  
\end{proof}

Theorem \ref{thm8pt4} 
indicates one reason why we should regard, in the case of discrete 
maximal surfaces, all FPS vertices as singular.  In fact, like in the case of 
discrete minimal surfaces, Lemma 
\ref{smoothMSparsurfs-max} indicates we can 
say the same of parallel surfaces of discrete maximal surfaces as well: 

\begin{defn}\label{defn:maxsurfparallelcase}
For any discrete maximal surface, we say 
that $(m,n)$ is a {\em singular vertex} of any given parallel surface 
(allowing also for the initial maximal surface itself) if the principal curvatures of the parallel surface satisfy
\begin{eqnarray*} 
\kappa_{(m-1,n),(m,n)} \cdot \kappa_{(m,n),(m+1,n)} \leq  0  \ \ \text{or} \\
\kappa_{(m,n-1),(m,n)} \cdot \kappa_{(m,n),(m,n+1)} \leq  0 \ . 
\end{eqnarray*}
\end{defn}

\begin{figure}
\begin{center}
\includegraphics[height=70mm]{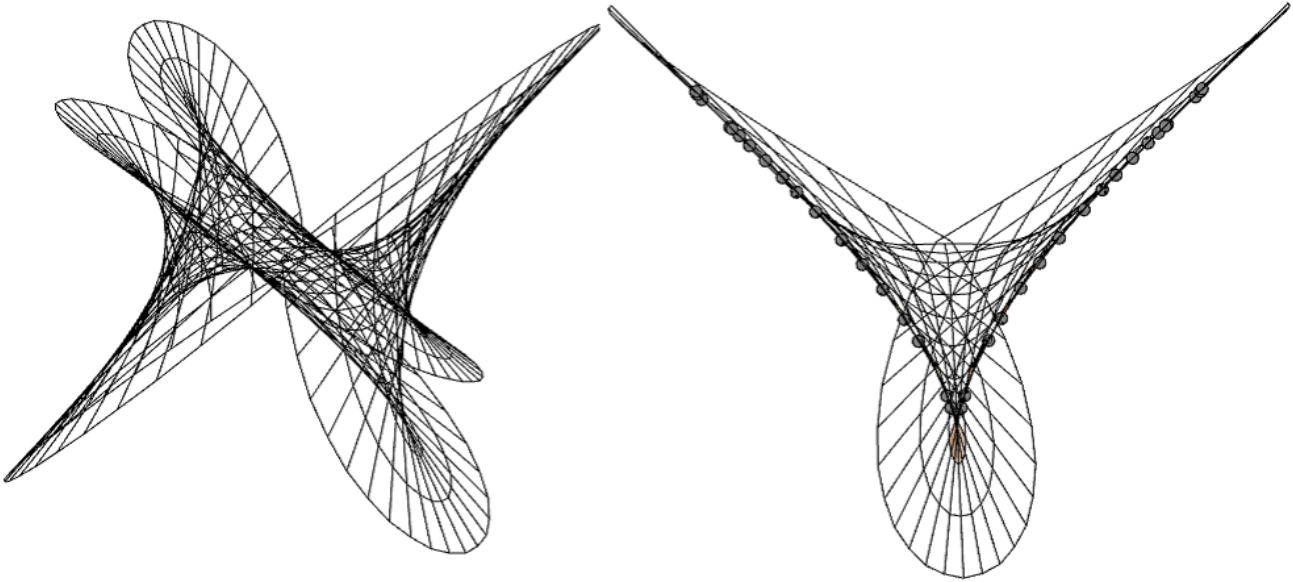} \\
\includegraphics[height=70mm]{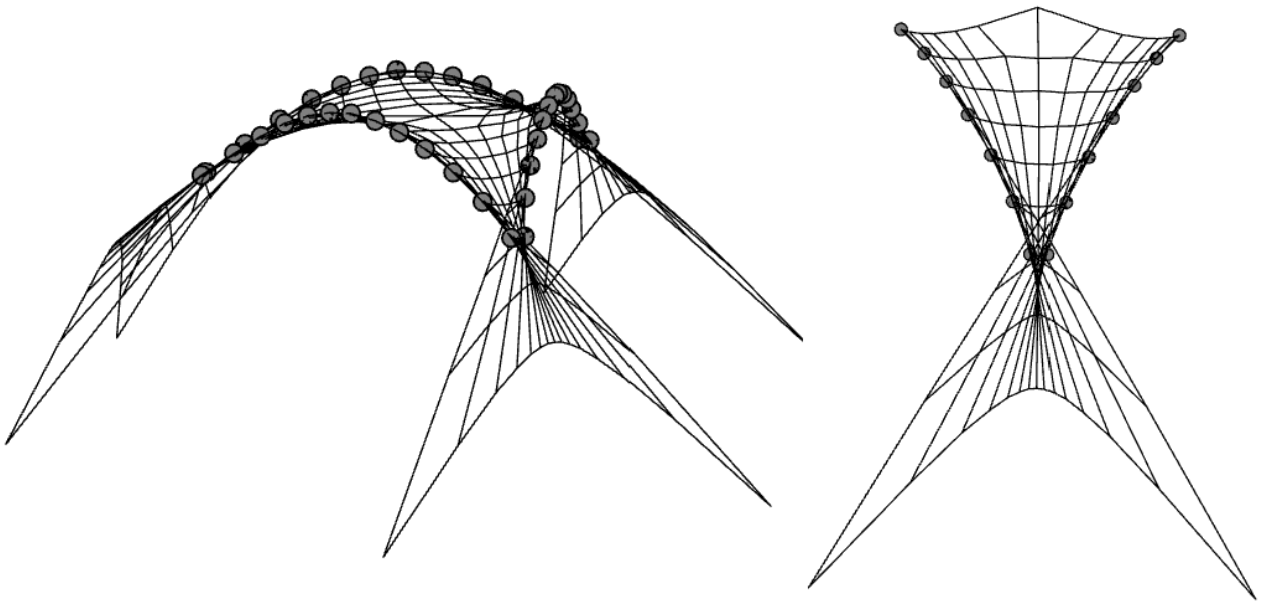}
\end{center}
\caption{A discrete higher-order Enneper-type maximal surface in $\mathbb{R}^{2,1}$, its parallel surface at distance $20$, and a one-third piece of the parallel surface, with singular vertices marked. }
\end{figure}

\section{Singular faces on discrete CMC $1$ surfaces with Weierstrass-type representations in $\mathbb{S}^{2,1}$}\label{secnine}

As in Definition \ref{defnOfSingFace}, a quadrilateral of a discrete CMC $1$ 
surface $\hat n$ with Weierstrass-type representation in $\mathbb{S}^{2,1}$ is singular if it does not lie in a 
spacelike plane. We give a geometric condition (Theorem \ref{thm7pt2next}) for when a 
quadrilateral of $\hat n$ is singular, analogous to a condition in 
the case of discrete maximal surfaces (see \cite{Yashi2014}).  We then prove a relation (Corollary \ref{SingFaceVert}) between FPS vertices and singular faces on 
discrete CMC $1$ faces in $\mathbb{S}^{2,1}$ (similar to Theorem 
\ref{thm8pt4}), a relation that helps indicate which of the FPS vertices are actually singular.  

The condition for a singular face to occur is 
\begin{equation}\label{eqn:detcond}
(d\hat f_{ij},d\hat f_{ij})(d\hat f_{i\ell},d\hat f_{i\ell})-
(d\hat f_{ij,}d\hat f_{i\ell})^2 \leq 0 \; . 
\end{equation}

In the smooth CMC $1$ case, with $g$ as in \S \ref{BrBi}, the singularities occur exactly where $|g|=1$.  The condition is still $|g|=1$ even under the coordinate transformation $z \to \sqrt{\lambda \alpha} z$. The next proposition and theorem are the corresponding condition in the discrete case to $|g|=1$, and can be proven by 
computationally spelling out Equation \eqref{eqn:detcond}.  
We define
\begin{eqnarray*}
h_1 = (1-|g_j|^2) |dg_{i\ell}|^2 (1-\lambda \alpha_{ij}) \; , \\
h_2 = (1-|g_\ell|^2) |dg_{ij}|^2 (1-\lambda \alpha_{i\ell}) \; , \\ 
h_3 = (1-|g_i|^2) |dg_{j\ell}|^2 \; . 
\end{eqnarray*}

\begin{proposition}\label{thm7pt2}
A face of a discrete 
CMC $1$ surface $\hat n$ with Weierstrass-type representation in $\mathbb{S}^{2,1}$ is singular if and only if 
\[ 
\mathcal{H}:=h_1^2+h_2^2+h_3^2-(h_2-h_1)^2-(h_3-h_1)^2-(h_3-h_2)^2 \leq 0 \; . 
\]
\end{proposition}

\begin{figure}
\begin{center}
\includegraphics[width=120mm]{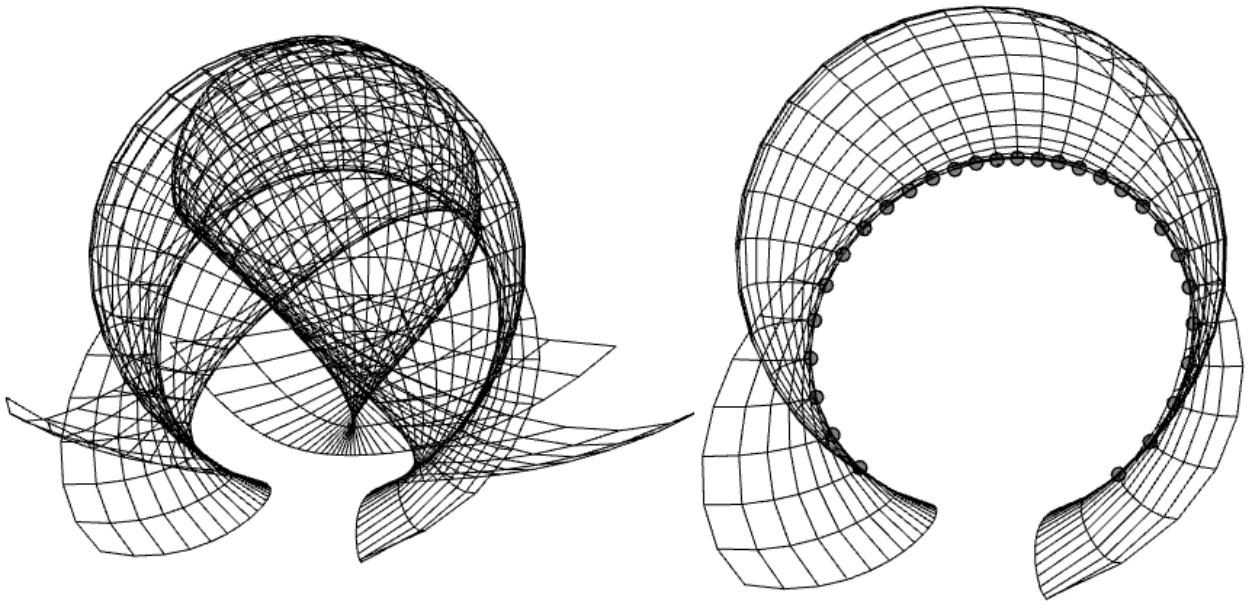} \\
\hspace{-10mm} \includegraphics[width=120mm]{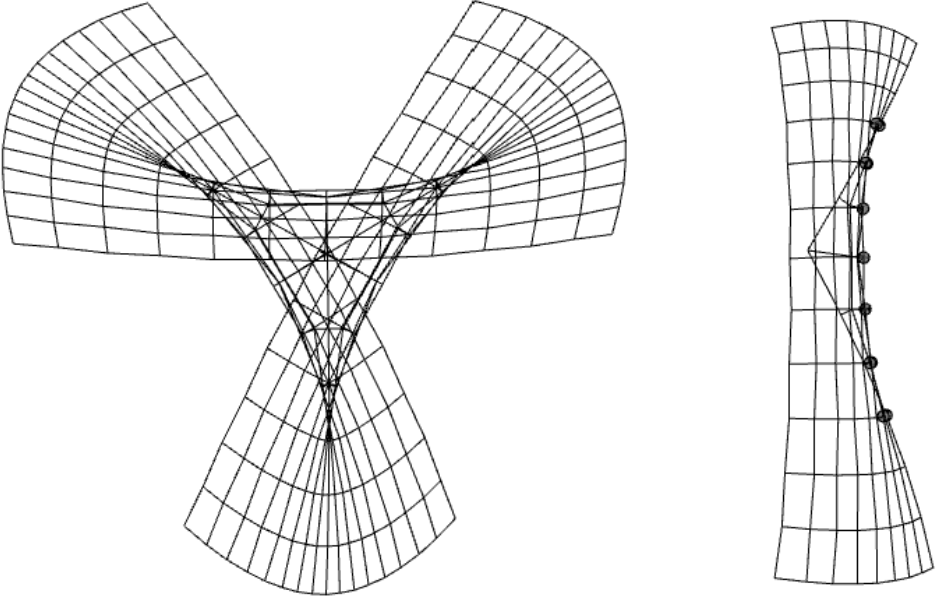} \\
\includegraphics[width=150mm]{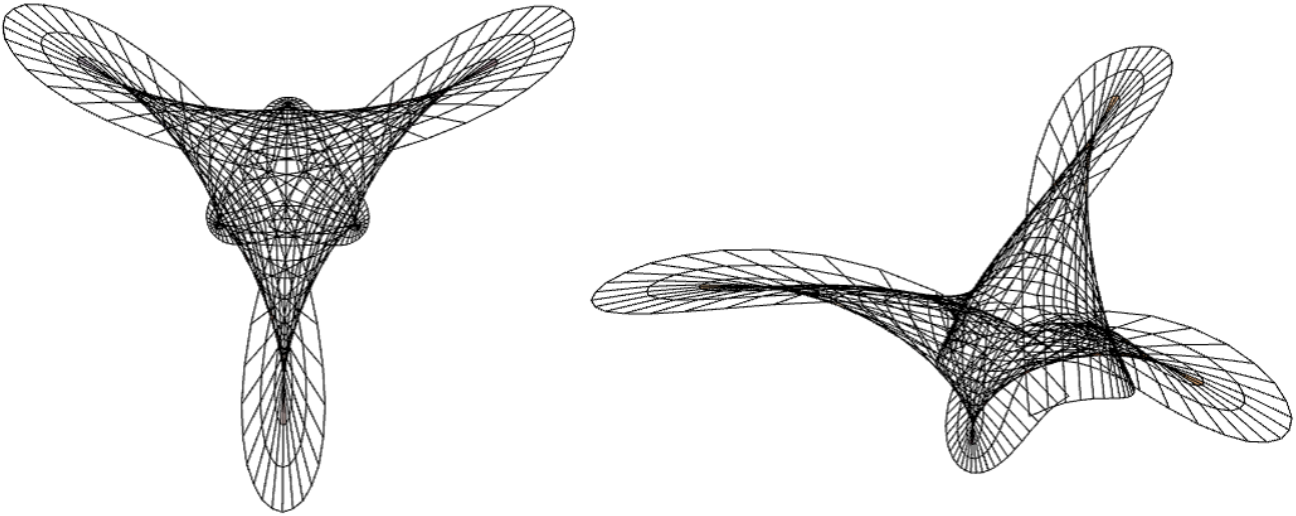}
\end{center}
\caption{From top to bottom: A discrete harmonic mean curvature $1$ surface in $\mathbb{S}^{2,1}$, a discrete flat surface in $\mathbb{S}^{2,1}$ and a discrete CMC 1 surface in $\mathbb{S}^{2,1}$, each shown twice. In order to draw the surfaces, we project to the hollow ball model for $\mathbb{S}^{2,1}$.  (For an explanation of the hollow ball model, see \cite{Fuji} for example.) }
\end{figure}

\begin{theorem}\label{thm7pt2next}
A quadrilateral of $\hat n$ as in Proposition \ref{thm7pt2} is singular for all 
$\lambda$ sufficiently close to zero if the 
corresponding circumcircle of $g$ intersects $\mathbb{S}^1$ 
transversally. The converse holds as well under the generic assumption that $\partial _{\lambda} \mathcal{H} \neq 0$.  
\end{theorem}

\begin{proof}
If the four points $g_i, \ g_j, \ g_k, \ g_l$ lie on a circle with radius $r \in \mathbb{R}$ and center $p \in \mathbb{C}$, the condition for $\mathcal{H}<0$ at $\lambda =0$ is 
\[
(|p|^2 -(r-1)^2) (|p|^2 -(r+1)^2) <0 \ .
\]
The result follows.
\end{proof}

\begin{theorem}\label{cor:9pt3}
Let $p_-$, $p$ and $p_+$ be three consecutive vertices in one 
direction in the lattice domain of a 
CMC $1$ surface $\hat{n}$ with Weierstrass-type representation in $\mathbb{S}^{2,1}$, 
with corresponding values $g_-$, 
$g$ and $g_+$ for the discrete holomorphic function in the 
Weierstrass type representation. Under the genericity assumption $|g| \neq  1$, then $\kappa_{p_-p} \cdot 
\kappa_{pp_+} < 0$ for all $\lambda$ sufficiently close to zero if and 
only if exactly one of $|g_-|^2$ and $|g_+|^2$ has value less than 
$1$ and the other has value greater than $1$.  
\end{theorem}

\begin{proof}
Because the surface is CMC $1$ in $\mathbb{S}^{2,1}$, we have 
$t=-1$.  Then Equations \eqref{lemmaEqn:aboutkappa} and \eqref{extra3} imply the result.  
\end{proof}

\begin{figure}
\begin{center}
\includegraphics[width=100mm]{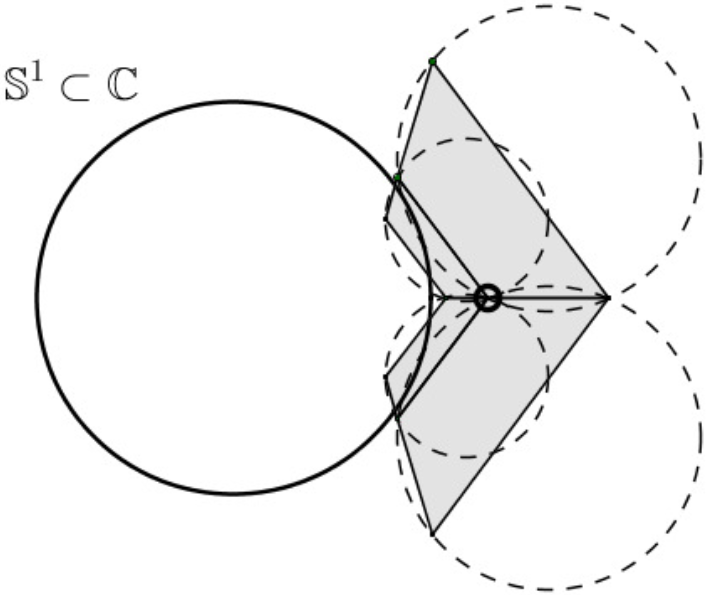}
\end{center}
\caption{A counterexample to the converse in Corollary \ref{SingFaceVert}. Numerical data for a discrete holomorphic function is shown. The four faces of a discrete CMC 1 surface determined from the four gray faces above are singular faces. On the other hand, for sufficiently small $\lambda$, the marked vertex is not singular. }
\label{fig7}
\end{figure}

This theorem tells us that we will find FPS vertices roughly 
where $g$ (discretely) crosses $\mathbb{S}^1$.  Because of 
Theorem \ref{cor:9pt3}, we can now regard 
these points as singular vertices and not parabolic nor flat points. Combining Theorems \ref{thm7pt2next} and \ref{cor:9pt3}, the following rigorous statement is immediate: 

\begin{corollary}\label{SingFaceVert}
Under the conditions of Theorem \ref{cor:9pt3}, 
for all $\lambda$ sufficiently close to zero, 
the pair of faces adjacent to the edge $p_-p$ are singular, or 
the pair of faces adjacent to the edge $pp_+$ are singular, 
including the possibility that all four faces are singular.  
\end{corollary}

The converse of this corollary does not hold, that is, it is possible 
to have four singular faces (for all $\lambda$ sufficiently close to $0$) adjacent to a given vertex that is non-singular for all $\lambda$ sufficiently close to $0$ (see Figure \ref{fig7}). Furthermore, taking $\lambda \to 0$, the example in Figure \ref{fig7} demonstrates that the converse to Theorem \ref{thm8pt4} also does not hold.

\section*{Acknowledgements}
The first author was partly supported by the Grant-in-Aid for Scientific Research (C) 15K04845 and (S) 24224001, Japan Society for the Promotion of Science. The second author was supported by the Grant-in-Aid for JSPS Fellows Number 26-3154, and the JSPS Program for Advancing Strategic International Networks to Accelerate the Circulation of Talented Researchers Mathematical Science of Symmetry, Topology and Moduli, Evolution of International Research Network based on OCAMI. Both authors were supported by the Grant-in-Aid for JSPS Fellows Number 26-3154, JSPS/FWF bilateral joint project ``Transformations and Singularities'' between Austria and Japan.


\end{document}